\documentclass{article}
\usepackage[english]{babel}
\usepackage{graphicx}
\usepackage{amsfonts}
\usepackage{amsmath}
\usepackage{amsthm}
\usepackage{float}
\usepackage{hyperref}
\usepackage{subfigure}
\usepackage[margin=1.5in]{geometry}
\usepackage[symbol]{footmisc}
\usepackage{footmisc}

\title{\vspace{-1.0cm}Cameron-Liebler line classes in AG($3,q$)}

\begin{document}

\theoremstyle{plain}
\newtheorem{St}{Theorem}[section]
\newtheorem{Le}[St]{Lemma}
\newtheorem{Gev}[St]{Corollary}
\newtheorem{Cons}[St]{Construction}
\theoremstyle{definition}
\newtheorem{Ex}[St]{Example}
\newtheorem{Def}[St]{Definition}
\newtheorem{Opm}[St]{Remark}

\newcommand{\gauss}[2]{\genfrac{[}{]}{0pt}{}{#1}{#2}}
\renewcommand{\thefootnote}{\fnsymbol{footnote}}

 \footnotetext{$*$ Department of Mathematics: Analysis, Logic and Discrete Mathematics, Ghent University, Krijgslaan 281, Building S8, 9000 Gent, Flanders, Belgium \\(Email: jozefien.dhaeseleer@ugent.be, leo.storme@ugent.be) (www: http://cage.ugent.be/$\sim$jmdhaese/, http://cage.ugent.be/$\sim$ls)\label{Gent}.}%
 \footnotetext{$\dagger$ Department of Mathematics and Data Science, University of Brussels (VUB),  Pleinlaan 2, Building G, 1050 Elsene, Brussels, Belgium \\ (Email: Jonathan.Mannaert@vub.be) (http://homepages.vub.ac.be/$\sim$jonmanna/)\label{VUB}.}%
\footnotetext{$\ddagger$ Department of Mathematics, University of Rijeka, Radmile Matej\v{c}i\'{c} 2, 51000 Rijeka, Croatia  \\(Email: asvob@math.uniri.hr)\label{Croatia}.}%

\author{J. D'haeseleer\footnotemark\label{Gent}, J. Mannaert\footnotemark{\label{VUB}}, L. Storme$^*$ and A. \v Svob\footnotemark\label{Croatia}}

\maketitle

\begin{abstract}
		The study of Cameron-Liebler line classes in PG($3,q$) arose from classifying specific collineation subgroups of PG($3,q$). Recently, these line classes were considered in new settings. In this point of view, we will generalize the concept of Cameron-Liebler line classes to AG($3,q$). In this article we define Cameron-Liebler line classes using the constant intersection property towards line spreads.  The interesting fact about this generalization is the link these line classes have with Cameron-Liebler line classes in PG($3,q$). Next to giving this link, we will also give some equivalent ways to consider Cameron-Liebler line classes in AG($3,q$), some classification results and an example based on the example found in \cite{DeBeule} and \cite{Feng}.
		
\end{abstract}	
	
	\section{Introduction}
	Let $p$ be a prime and $q=p^h$, with $h\geq 1$. Then we can consider PG($3,q$) and the corresponding affine space AG($3,q$), with $\pi_\infty$ the hyperplane at infinity, as the $3$-dimensional projective and affine space over $\mathbb{F}_q$.
The fact that these spaces are linked, will lead to interesting connections in the study of Cameron-Liebler line classes. We start with the definition of some line sets in both AG($3,q$) and PG($3,q$).
\begin{Def}\label{DefSpread}
Consider PG($3,q$) or AG($3,q$).
\begin{enumerate}
\item A set of pairwise disjoint lines is called a \emph{partial line spread}.
\item A pair of \emph{conjugated switching sets} consists of two disjoint partial line spreads that cover the same set of points.
\item A \emph{line spread} is a partial line spread that partitions the points of the corresponding space. The size of such line spreads is fixed. In PG($3,q$), a line spread has size $q^2+1$ and in AG($3,q$) it has size $q^2$.
\end{enumerate}
\end{Def}

	
	In the following lemma we describe some examples of line spreads in AG($3,q$).
	\newpage

	\begin{Le}\label{SpreadsAG(3,q)}
			Consider the affine space AG($3,q$) and the corresponding projective space PG($3,q$). Then the following sets $\mathcal{S}$ are line spreads in AG($3,q$).
			
			\begin{enumerate}
				\item (\emph{Type I}) Let $\mathcal{S}$ be a line spread in PG($3,q$), then the restriction of lines of $\mathcal{S}$ in AG($3,q$) is  a line spread in this space. This restriction of $\mathcal{S}$ to AG($3,q$) consists of all the lines of $\mathcal{S}$ except the unique line at infinity in $\mathcal{S}$.
				\item (\emph{Type II}) Consider a point \(p\) in \(\pi_\infty\) and define	the set \(\mathcal{S}\) as the set of all affine lines through $p$.	
			 \end{enumerate}
	\end{Le}
 \begin{proof}
 	These examples are indeed line spreads in AG($3,q$).

\end{proof}
\begin{Opm}
It is possible to define a generalization of line spreads of type II. Consider AG($3,q$) and let \(\ell\) be a line in the plane at infinity \(\pi_\infty\). Then there are exactly \(q\) planes through $\ell$ not equal to \(\pi_\infty\). If we call these planes \(\pi_i\), for \(i \in \{1,...,q\}\), then we can choose for every plane \(\pi_i\) a point \(p_i\) on $\ell$. This defines the following line spread
				 \[\mathcal{S}:=\{\ell \mid  \ell \text{ a line, such that  } p_i\in \ell \subseteq \pi_i \text{ for some } i \in \{1,...,q\}\}.\]	
It is clear that if we would choose all points $p_i$ equal, we obtain a line spread of type II.
\end{Opm}

\section{Preliminary results}

In this section, we will give some useful results that will be used later on. We will give these results as general as possible so that they can be used in other contexts. First of all we choose $n>1$ and $1\leq k\leq n-1$. Here we consider the affine space AG($n,q$) and the corresponding projective space PG($n,q$). In these spaces, we define $\Pi_k$, and $\Phi_k$, as the set of $k$-spaces in PG($n,q$), and AG($n,q$), respectively. 

\begin{Cons}[Incidence matrix]\label{IncidenceMatrix}

Consider the incidence matrix \(P_n\) of PG($n,q$), defined over the field $\mathbb{C}$, where the rows correspond to the points and the columns correspond to the elements of \(\Pi_k\). We order the rows and columns in such a way that the first rows and columns correspond to the affine points and $k$-spaces respectively.  Then \(P_n\) is of the following form:
\begin{equation}\label{IncidenceMatrixPn}
P_n= \begin{bmatrix}
A & \bar{0} \\
B_2 & P_{n-1}
\end{bmatrix}.
\end{equation}
Here \(A\) is the incidence matrix of AG($n,q$), where again the rows  correspond to the points and the columns correspond to the elements of \(\Phi_k\). The matrix $\bar{0}$ is the zero-matrix and the part that remains unnamed, we call \(B_2\).
\end{Cons}
Before we state an important result, we will give a lemma that will be useful.
\begin{Le}(\cite[Theorem 9.5]{DeBruyn})\label{FullRankIncidence}
The point-($k$-space) incidence matrix of PG($n,q$) or AG($n,q$) has full rank.
\end{Le}
Using this lemma, we give the following important result.
\begin{St}\label{chiUpToDown}
	Consider a set \(\mathcal{L}\) of $k$-spaces in PG($n,q$) and let $P_n$ and $A$ be as in Construction \ref{IncidenceMatrix}. If the characteristic vector \(\widetilde{\chi}_{\mathcal{L} }$ lies in $ (\ker(P_n))^\perp \) and if \(\mathcal{L}\)  contains no $k$-spaces at infinity, then  \(\widetilde{\chi}_{\mathcal{L} } \) restricted to AG($n,q$) belongs to \( (\ker(A))^\perp\).
\end{St}
\begin{proof}
	We consider the characteristic vector $\widetilde{\chi}_{\mathcal{L} }$ of the set \(\mathcal{L}\). Since this set misses the hyperplane at infinity, it follows from our chosen ordering in Construction \ref{IncidenceMatrix} that $\widetilde{\chi}_{\mathcal{L} }$ is of the form:
	\[\widetilde{\chi}_{\mathcal{L} }=\begin{pmatrix}
	\chi_{\mathcal{L}} \\
	\bar{0}
	\end{pmatrix} \in (\ker(P_{n}))^\perp.\]
	Consider now a vector \(v_1 \in \ker(A)\), then we need to prove that \(\chi_{\mathcal{L}}\cdot v_1=0\). We claim that we are able to find a vector \(v_2\), such that \(v= \begin{pmatrix}
	v_1 \\
	v_2
	\end{pmatrix} \in \ker(P_{n})\). For this vector \(v\) we find that
	\[0=\widetilde{\chi}_{\mathcal{L} }\cdot v= \chi_{\mathcal{L}}\cdot v_1,\]
	and since \(v_1\) was arbitrarily chosen, the proof is done. But we still need to prove our claim. So for \(v_1\) as above, we need to find a vector \(v_2\) such that
	\begin{equation}\label{CUDEq}
	\left\{
	\begin{array}{ll}
	Av_1+\bar{0}v_2 & =0   \\
	B_2v_1+P_{n-1}v_2 &=0 
	\end{array} \right. ,
	\end{equation}
holds. Notice that the left-hand side of line 1 in Equation (\ref{CUDEq}) is always zero since \(v_1 \in \ker(A)\). 
Due to Lemma \ref{FullRankIncidence}, the matrix \(P_{n-1}\) has full rank. Hereby, we can always find a vector \(v_2\) that makes the left-hand side  of line 2 in Equation (\ref{CUDEq}) zero.
\end{proof}

A last result that we want to state is the following. For this result we need to define $k$-spreads in AG($n,q$).  A $k$-spread in AG($n,q$) is a set of skew $k$-spaces that partitions the point set of AG($n,q$).
\begin{Le}\label{2to8}
	Consider a set \(\mathcal{L}\) of $k$-spaces in AG($n,q$), such that \(\chi_{\mathcal{L}} \in (\ker(A))^\perp\) with $A$ equal to the point-($k$-space) incidence matrix of AG($n,q$). Then it follows that for every affine $k$-spread \(\mathcal{S}\), it holds that
	\[|\mathcal{L}\cap\mathcal{S}|=x,\]
	for a certain fixed integer \(x\).
\end{Le}
\begin{proof}
	Let  \(\chi_\mathcal{S}\) be the characteristic vector of the $k$-spread $\mathcal{S}$, then
	\begin{equation} \label{2.4}\chi_{\mathcal{S}}- \gauss{n}{k}_q^{-1}\textbf{j} \in \ker(A).\end{equation}
Here $\gauss{n}{k}_q$ equals the number of $k$-spaces through a point in AG($n,q$) and $\textbf{j}$ is the vector of the correct dimension that contains $1$ on every position.
	 Equation (\ref{2.4}) is valid since every point is contained in exactly one element of \(\mathcal{S}\), so every row of \(A\) and \(\chi_{\mathcal{S}}\) have exactly one 1 in common. The assertion then follows since from \(\chi_{\mathcal{L}}\in (\ker(A))^\perp\), we know that
	\[\chi_{\mathcal{L}}\cdot \left( \chi_{\mathcal{S}}- \gauss{n}{k}_q^{-1}\textbf{j}\right) =0,\]
	or simplified 
	\[|\mathcal{L}\cap\mathcal{S}|=\gauss{n}{k}_q^{-1}|\mathcal{L}|=x.\]
	Here $x$ is indeed an integer.
\end{proof}

\section{Cameron-Liebler line classes}
		
	Cameron-Liebler line classes were first observed in \cite{Cameron-Liebler}. In this article, Cameron and Liebler tried to classify the collineation subgroups of PG($3,q$) that have the same number of orbits on the lines as on the points. They noticed that these orbits on the lines have special properties. Line classes that satisfy these properties were later called Cameron-Liebler line classes. We will give the definition of a Cameron-Liebler line class for PG($3,q$) as well as for AG($3,q$).

\begin{Def}\label{CLLineClass}
A \emph{Cameron-Liebler line class}, or briefly CL line class, of parameter $x$ in both PG($3,q$) or AG($3,q$) is a set of lines, such that for every line spread $\mathcal{S}$ it holds that $|\mathcal{L} \cap \mathcal{S}|=x$.
\end{Def}

If we now recall Lemma \ref{SpreadsAG(3,q)}, we obtain that every line spread of PG($3,q$) without its line at infinity is a line spread of AG($3,q$). Hence, since CL line classes in AG($3,q$) do not have lines at infinity, we find that CL line classes in AG($3,q$) are in fact special types of CL line classes in PG($3,q$). We thus have the following theorem.

\begin{St}\label{AffCLBasic}
	If  \(\mathcal{L}\) is a CL line class of AG($3,q$), then \(\mathcal{L}\) is a CL line class in the corresponding projective space PG($3,q$) with the same parameter $x$.
\end{St}
\begin{proof}
Follows from Lemma \ref{SpreadsAG(3,q)} (1).
\end{proof}

This will lead to a translation of the known results for CL line classes in PG($3,q$) to CL line classes in AG($3,q$). Another consequence of  Lemma \ref{SpreadsAG(3,q)} is that AG($3,q$) has significantly more line spreads than PG($3,q$), which will result in stronger existence conditions. This will rule out the parameter  $x=2$, which we will prove in Corollary \ref{SmallClas}.
First we focus on some basic and already known results of CL line classes in AG($3,q$) and PG($3,q$) respectively.

	 \begin{St}(\cite[Theorem 1]{Penttila1991} and \cite[Definition 1.1]{Leo}) \label{ProjEqLines}
	 	Let  \(\mathcal{L}\)  be a set of lines of size \(|\mathcal{L}|=x(q^2+q+1)\) in PG($3,q$) with characteristic vector \(\chi_{\mathcal{L}}\). Then the following properties are equivalent:
	 	\begin{enumerate}
	 		\item  Let \(P_3\) be the point-line incidence matrix of PG($3,q$), then $\chi_{\mathcal{L}} \in (\ker(P_3))^\perp= $Im($P_3^T$).
	 		\item For every line spread \(\mathcal{S}\), \(|\mathcal{L}\cap\mathcal{S}|=x\).
	 		\item For every pair of conjugated switching sets \(\mathcal{R}\) and \(\mathcal{R}'\), \(|\mathcal{L}\cap\mathcal{R}|= | \mathcal{L} \cap \mathcal{R}'|.\)
	 		\item For every line $\ell$, there are exactly \((x-\chi_{\mathcal{L}}(\ell))q^2\) elements of \(\mathcal{L}\) disjoint to $\ell$.  Here $\chi_{\mathcal{L}}(\ell)=1$ if $\ell\in\mathcal{L}$ and zero otherwise.
	 		\item For every point $p$ and plane $\pi$, with $p \in \pi$, 
	 		\[| star(p)\cap \mathcal{L}| + |line(\pi)\cap\mathcal{L}|=x +(q+1)|pencil(p,\pi)\cap \mathcal{L}|.\]
			Here $star(p)$ is the set of lines through the point $p$, $line(\pi)$ is the set of lines in the plane $\pi$ and $pencil(p,\pi)=star(p)\cap line(\pi)$.
	 	\end{enumerate}
	 \end{St}
	Some trivial examples of CL line classes in PG($3,q$) are the following:
\begin{Ex}\cite[Proposition 3.4]{Cameron-Liebler}\label{TrivEx}
The following sets are CL line classes in PG($3,q$).
\begin{enumerate}
\item The empty set, which has parameter $x=0$.
\item  All the lines through a fixed point $p$. This example has parameter $x=1$.
 \item All the lines in a fixed plane. Also this example has parameter $x=1$ and is in fact the dual of the previous example. 
\item  Take a plane $\pi$ and a point $p \not \in\pi$, then the set of all the lines through $p$ together with all the lines in $\pi$ gives a CL line class of parameter $x=2$.
\end{enumerate}
\end{Ex}
\begin{St}\cite[Proposition 3.4]{Cameron-Liebler}\label{TrivExSt}
If $\mathcal{L}$ is a CL line class of parameter $x \in \{0,1,2\}$ in PG($3,q$), then $\mathcal{L}$ is listed in Example \ref{TrivEx}.
\end{St}

It was long thought that these examples, together with their complements, were the only  CL line classes. This was disproven by Drudge in \cite{Drudge} who found an example of parameter $x=5$ in PG($3,3$). This example was later generalized  in \cite{BruenAndDrude} to an infinite family of parameter $x=\frac{q^2+1}{2}$ in PG($3,q$), for $q$ odd. More recently, another infinite family was found simultaneously in \cite{DeBeule} and \cite{Feng}, and has parameter $x=\frac{q^2-1}{2}$ in PG($3,q$), with $q \equiv 5$ or $9 \pmod{12}$.
	
We now list some consequences and basic results on CL line classes in AG($3,q$).
\begin{Le}\label{size}
	For every CL line class $\mathcal{L}$ in AG($3,q$), it holds that \(|\mathcal{L}|=x(q^2+q+1)\).
\end{Le}
\begin{proof}
	From the definition of a CL line class in AG($3,q$), we know that $\mathcal{L}$ contains $x$ lines of every line spread of type II. Since there are $q^2+q+1$ such line spreads all consisting of disjoint line sets, the assertion follows.
\end{proof}
The analog goes for the following equivalences, which hold for as well PG($3,q$) as for AG($3,q$).

\begin{Le}\label{PropCLLC}\label{ComplementCLLC}
Consider CL line classes $\mathcal{L}$ and $\mathcal{L}'$ with parameter $x$ and $x'$ (in PG($3,q$) or AG($3,q$)), then the following statements are true.
\begin{enumerate}
\item For the parameter $x$, $0\leq x \leq q^2+1$ in PG($3,q$) or $0 \leq x\leq q^2$ in AG($3,q$).
\item If $\mathcal{L}'\subseteq \mathcal{L}$, then $\mathcal{L}\setminus \mathcal{L}'$ is a CL line class of parameter $x-x'$. Consequently, the complement of $\mathcal{L}$ in AG($3,q$) and PG($3,q$) is also a CL line class of parameter $q^2-x$ and $q^2+1-x$, respectively.
\item If $\mathcal{L} \cap \mathcal{L}'=\emptyset$, then $\mathcal{L}\cup \mathcal{L}'$ is a CL line class of parameter $x+x'$.
\end{enumerate}
\end{Le}
\begin{proof}
Property (1) follows from the size of a line spread and Definition \ref{CLLineClass}. The other properties also follow from the same definition.
\end{proof}

Theorem \ref{AffCLBasic} gives one connection between CL line classes in AG($3,q$) and in PG($3,q$). We now give a converse connection. 

\begin{St}\label{CLAf3}
	Suppose that \(\mathcal{L}\) is a CL line class in PG($3,q$) of parameter $x$. Then \(\mathcal{L}\) defines a CL line class in AG($3,q$) with the same parameter $x$ if and only if \(\mathcal{L}\) is disjoint to the set of lines in the plane at infinity of AG($3,q$).
\end{St}
\begin{proof}
Let $\mathcal{L}$ be a CL line class with parameter \(x\) in PG($3,q$) that is disjoint to the set of lines in the plane at infinity. Then we can consider $A$ and $P_3$ from Construction \ref{IncidenceMatrix} for $k=1$. These are the point-line incidence matrices of AG($3,q$) and PG($3,q$) respectively. Secondly, we  write the characteristic vector of $\mathcal{L}$ corresponding to the projective space as \(\widetilde{\chi}_{\mathcal{L}}\) and the characteristic vector corresponding to the affine space as \(\chi_{\mathcal{L}}\). Since we know that \(\mathcal{L}\) is disjoint to the set of lines at infinity and \(\mathcal{L}\) is a CL line class in PG($3,q$), we have that

	\[\widetilde{\chi}_{\mathcal{L}}=\begin{pmatrix}
	\chi_{\mathcal{L} }\\
	\bar{0}
	\end{pmatrix} \in (\ker(P_{3}))^\perp.\]
	Using Theorem \ref{chiUpToDown}, we get that  \(\chi_{\mathcal{L} } \in (\ker(A))^\perp\). So due to Lemma \ref{2to8}, we find that \(\mathcal{L}\) is also a CL line class in AG($3,q$). Since \(|\mathcal{L}|=(q^2+q+1)x\) as a projective CL line class, we know by using  Lemma \ref{size} that $\mathcal{L}$  also has parameter $x$ in AG($3,q$).
	
	We now prove the converse. Let \(\mathcal{L}\) be a projective CL line class of parameter $x$, such that its restriction to the affine space also defines a CL line class \(\mathcal{L}'\) of parameter $x$. Then by Theorem \ref{AffCLBasic}, it follows  that this restriction considered in PG($3,q$) is also a projective CL line class  of parameter $x$. The line set \(\mathcal{L}'\) will contain no lines at infinity and we have that \(\mathcal{L}'\subseteq \mathcal{L}\). Due to the fact that both have size $x(q^2+q+1)$, we obtain that \(\mathcal{L}=\mathcal{L}'\) and does not contain lines in \(\pi_\infty\).
\end{proof}
\begin{Opm}
Suppose that $\mathcal{L}$ is a CL line class with parameter $x+1$ in PG($3,q$), that contains all the lines in the plane at infinity. Then we can use Example \ref{TrivEx} and Lemma \ref{PropCLLC} to obtain a CL line class with parameter $x$ in PG($3,q$) disjoint to the set of lines at infinity. Here we can use the previous theorem and we see a $1-1$ connection between these projective CL line classes and CL line classes in AG($3,q$) of parameter $x$. 
\end{Opm}

\section{Equivalent definitions and non-existence conditions}
There are a lot of equivalent definitions for CL line classes in PG($3,q$). We can ask, keeping Theorem \ref{AffCLBasic} in mind, if they correspond to equivalent definitions in the affine case. 

\begin{St}\label{equivCLAff}
	Consider in the affine space AG($3,q$) a set of lines \(\mathcal{L}\) such that \(|\mathcal{L}|=x(q^2+q+1)\), with \(x\) a positive integer. Let $A$ be the point-line incidence matrix of AG($3,q$), then the following properties are equivalent.
	\begin{enumerate}
		\item For every line spread $\mathcal{S}$, $|\mathcal{L}\cap \mathcal{S}|=x$.
		\item The characteristic vector  \(\chi_{\mathcal{L}}\in (\ker(A))^\perp\).
		\item The characteristic vector \(\chi_{\mathcal{L}}\in\) Im\((A^T)\).
		\item For every pair of conjugated switching sets \(\mathcal{R}\) and \(\mathcal{R}'\), \(|\mathcal{L}\cap\mathcal{R}|=|\mathcal{L}\cap\mathcal{R}'|\).
		
		\item  For every line \(\ell\), the number of elements of \(\mathcal{L}\) disjoint to \(\ell\) is equal to
		\[(q^2+1)(x-\chi_{\mathcal{L}}(\ell)),\]
		and through every point at infinity there are exactly \(x\) lines of \(\mathcal{L}\). Here $\chi_{\mathcal{L}}(\ell)=1$ if $\ell\in\mathcal{L}$ and zero otherwise.
	\end{enumerate}
\end{St}

\begin{proof}
	First of all it is clear that (2) is equivalent with (3), since this is the case for every matrix.
	\begin{enumerate}
		\item From (1) to (2): Note that $\mathcal{L}$ is by definition a CL line class in AG($3,q$). So from Theorem \ref{AffCLBasic}, we know that \(\mathcal{L}\) defines a CL line class in the corresponding projective space. 
Here we can use Theorem \ref{ProjEqLines}  to obtain that the characteristic vector corresponding to \(\mathcal{L}\) in PG($3,q$) lies in \((\ker(P_3))^\perp\), with \(P_3\) the point-line incidence matrix of PG($3,q$) (see Construction \ref{IncidenceMatrix}). 
Since $\mathcal{L}$ does not contain lines in \(\pi_\infty\), we may use Theorem \ref{chiUpToDown}. Thus  we obtain that the characteristic vector restricted to the affine space lies in \((\ker(A))^\perp\). But this restriction is exactly  \(\chi_{\mathcal{L}}\).

\item From (2) to (4): 	Since \(\mathcal{R}\) and \(\mathcal{R}'\) are conjugated switching sets, they cover the same set of points, so it necessarily holds for the characteristic vectors that
		\[\chi_{\mathcal{R}}-\chi_{\mathcal{R}'}\in \ker (A).\]
		This implies that
		\[\chi_{\mathcal{L}}\cdot(\chi_{\mathcal{R}}-\chi_{\mathcal{R}'})=0.\]
		From here we find that
		\[|\mathcal{L}\cap\mathcal{R}|-|\mathcal{L}\cap\mathcal{R}'|=0,\]
		which proves the statement.
		
		\item From (4) to (1): 	Consider two spreads \(\mathcal{S}_1, \mathcal{S}_2\). Then we know that \(\mathcal{S}_1\setminus \mathcal{S}_2\) and \(\mathcal{S}_2\setminus \mathcal{S}_1\) are conjugated switching sets, since they cover the same set of points and have no elements in common. So, by (4), we know that
		\[|\mathcal{L}\cap(\mathcal{S}_1\setminus\mathcal{S}_2)|=|\mathcal{L}\cap(\mathcal{S}_2\setminus\mathcal{S}_1)|\]
		which implies that 
		\[|\mathcal{L}\cap\mathcal{S}_1|=|\mathcal{L}\cap\mathcal{S}_2|=c. \]
But we still have to prove that \(c=x\).
By definition, we obtain that $\mathcal{L}$ is a CL line class in AG($3,q$) with parameter $c$. Hence, we obtain that $|\mathcal{L}|=c(q^2+q+1)$, such that, by assumption, we have that $c=x$.
	\end{enumerate}
	
	Till here we already have proven that the first four points are equivalent, so we only need to prove the equivalence with property (5).
	
	First if Property (1) holds, then \(\mathcal{L}\) is a CL line class in AG($3,q$) with parameter $x$. So, by Theorem \ref{AffCLBasic}, this line class is also a CL line class in the corresponding projective space with the same parameter $x$. Due to Theorem \ref{ProjEqLines},  we know that there are exactly \(q^2(x- \chi_{\mathcal{L}}(\ell))\) lines of \(\mathcal{L}\) disjoint to \(\ell\) in PG($3,q$). Since \(\mathcal{L}\) is an affine line set, these \(q^2(x- \chi_{\mathcal{L}}(\ell))\) lines are all lines of $\mathcal{L}$ that are disjoint to $\ell$ in AG($3,q$). So the only elements we still need to count are those lines of \(\mathcal{L}\) that are disjoint to $\ell$ in AG($3,q$) and intersect in PG($3,q$). Hence, these lines will intersect $\ell$ in the point $p=\ell \cap \pi_\infty$ at infinity. If we consider the line spread \(\mathcal{S}\) of type II containing all affine lines through $p=\ell \cap \pi_\infty$, then \(|\mathcal{L}\cap \mathcal{S}|=x \). From this we obtain the second part and that there are in total exactly \[(q^2+1)(x-\chi_{\mathcal{L}}(\ell))\]
elements of \(\mathcal{L}\) disjoint to \(\ell\) in AG($3,q$).
	
	We now prove the converse direction. Consider a set of lines \(\mathcal{L}\) such that Property (5) holds. Then we know that for every affine line $\ell$ in PG($3,q$), the number of lines of $\mathcal{L}$ that are disjoint to $\ell$ in PG($3,q$) is equal to  \(q^2(x-\chi_{\mathcal{L}}(\ell))\). This follows from the fact that we subtracted those $x-\chi_{\mathcal{L}}(\ell)$ lines that intersect $\ell$ at infinity. 
	But if we now look at a line \(\ell'\subseteq \pi_\infty\), then we know that there are \(q^2\) points in \(\pi_\infty\setminus \ell'\). Through each of those points we have \(x\) affine lines of \(\mathcal{L}\), which are all disjoint to \(\ell'\) in PG($3,q$). These are also all the disjoint lines of \(\mathcal{L}\). Since if we had another line of \(\mathcal{L}\) that is disjoint to \(\ell'\), it should first be an affine line that then intersects \(\pi_\infty\) in a point. This implies that we in fact already counted it. We get that there are \(q^2x\) elements of \(\mathcal{L}\) disjoint to \(\ell'\not \in \mathcal{L}\). 

So we conclude that for any arbitrary projective line \(\ell\) in PG($3,q$), there are exactly
	\(q^2(x-\chi_{\mathcal{L}}(\ell))\)
	lines of \(\mathcal{L}\) disjoint to \(\ell\). This is equivalent with definition (5) in Theorem \ref{ProjEqLines} of a CL line class in the projective space. Thus \(\mathcal{L}\) defines a CL line class in PG($3,q$). Since \(\mathcal{L}\) is a line set that is defined in AG($3,q$), we know that \(\mathcal{L}\) contains no lines in the plane at infinity. Using Theorem \ref{CLAf3}, we have that \(\mathcal{L}\) is a CL line class in AG($3,q$), where the parameter $x$ follows from its size.
\end{proof}

For the last part of this paper, we want to give some non-existence results and some examples. Let us first state a more recent result of Gavrilyuk and Metsch in \cite{MetschAndGavrilyuk} about CL line classes in PG($3,q$).
\begin{St}(\cite[Theorem 1.1]{MetschAndGavrilyuk})\label{MetschRes}
	Suppose that \(\mathcal{L}\) is a CL line class with parameter \(x\) of PG($3,q$). Then for every plane and every point of PG($3,q$), 
	\begin{equation}\label{Metsch}
	\binom{x}{2} +n(n-x)\equiv 0 \mod (q+1),
	\end{equation}
	where \(n\) is the number of lines of \(\mathcal{L}\) in the plane, respectively through the point.
\end{St} 

If we translate this result with Theorem \ref{AffCLBasic} to AG($3,q$), it will look like this.

\begin{Gev}\label{GevKlausMain}
	If  \(\mathcal{L}\) defines a CL line class in AG($3,q$) with parameter $x$, then the equation
	\begin{equation}\label{CL1}
	\frac{x(x-1)}{2} \equiv 0 \mod (q+1)
	\end{equation}
	holds.
\end{Gev}
\begin{proof}
	Due to Theorem \ref{AffCLBasic}, every CL line class $\mathcal{L}$ in AG($3,q$) is a CL line class in PG($3,q$) of the same parameter $x$. So we may use Theorem \ref{MetschRes}. Here we notice that $\mathcal{L}$ is disjoint to the set of all lines in a plane, namely the plane at infinity. Thus we may fill in $n=0$.
\end{proof}

One could ask how good this non-existence condition is. What if we for example choose another plane or point in PG($3,q$), could we find more conditions by using Theorem \ref{MetschRes}. Notice first that choosing a point at infinity would not help, since this point defines a line spread of type II and thus would always contain $n=x$ lines of $\mathcal{L}$. This leads to the same result. But what if we choose another plane or point not at infinity? It can be proven that this is not helpful either. This is done in the following lemma.

\begin{Le}\label{NoBetter}
	Let $\mathcal{L}$ be a CL line class of parameter $x$ in AG($3,q$), then for every  plane $\pi$ it holds that 
	$$|line(\pi)\cap\mathcal{L}|\equiv 0 \mod{(q+1)}.$$
Consequently, it holds for every point $p$ that $|star(p)\cap\mathcal{L}|\equiv x \mod{(q+1)}.$
\end{Le}
\begin{proof}
	We consider the affine space AG($3,q$), together with the corresponding projective space PG($3,q$). We also define $\pi_\infty$ as the plane at infinity.
	Consider a CL line class $\mathcal{L}$ of parameter $x$ in AG($3,q$).
	Then, by Theorem \ref{AffCLBasic}, we know that $\mathcal{L}$ defines a CL line class in PG($3,q$) with the same parameter $x$. Here we can use Theorem \ref{ProjEqLines} to obtain that  for every point $p'$ and plane $\pi$, with $p' \in \pi$, we get that
	\begin{equation}\label{NoBetter1}
| star(p')\cap \mathcal{L}| + |line(\pi)\cap\mathcal{L}|=x +(q+1)|pencil(p',\pi)\cap \mathcal{L}|.
\end{equation}
	If we now choose $\pi$ as an arbitrary affine plane and $p'$ as a point at infinity. Then we get that $| star(p')\cap \mathcal{L}|=x$, thus we obtain that $|line(\pi)\cap\mathcal{L}|\equiv 0 \mod{(q+1)}$.

Secondly, pick an arbitrary point $p$ and choose an arbitrary plane $\pi'$ containing this point. Then we can consider Equation (\ref{NoBetter1}) modulo $q+1$ and obtain with our previously proven fact that 
$$|star(p)\cap \mathcal{L}|+0\equiv x+0  \mod{(q+1)}.$$
This proves the lemma.
\end{proof}
This lemma proves that choosing another plane or point in Theorem \ref{MetschRes} will not improve Corollary \ref{GevKlausMain}, since this would lead to the same equation by filling in $n=0$ or $n=x$. 

We are ready to give a classification of the CL line classes with small parameter $x$.
\begin{Gev}\label{SmallClas}
Consider the space AG($3,q$).
\begin{itemize}
	\item The only CL line class of parameter $x=0$ is the empty set.
	\item The only CL line class  of parameter $x=1$ is a point-pencil defined by an affine point.
	\item There are no CL line classes of parameter $x=2$.
\end{itemize}
\end{Gev}
\begin{proof}
The case for $x=0$ is trivial.
Suppose that $\mathcal{L}$ defines a CL line class of parameter $x$ in AG($3,q$). Then we know due to Theorem \ref{AffCLBasic} that $\mathcal{L}$ also defines a CL line class in PG($3,q$) of the same parameter $x$. So for $x=1$ we know that, due to Theorem \ref{TrivExSt}, $\mathcal{L}$ consists of all lines through a point or the set of all the lines in a plane. Since it is immediately clear that all the lines in a plane does not satisfy the constant intersection axiom with line spreads of type II, we may conclude that the only example for $x=1$ left is the set of all lines through an affine point.
	
	The case $x=2$ follows trivially from Corollary \ref{GevKlausMain}, since it is clear that $x=2$ does never satisfy $\frac{x(x-1)}{2}\equiv 0 \pmod{(q+1)}$.
\end{proof}
\begin{Opm}\label{HogeParameters}
	As a consequence of this Corollary and Lemma \ref{PropCLLC}, we know that in AG($3,q$) there do not exist CL line classes of parameter $x=q^2-2$. Another consequence is that the only possibility for $x=q^2-1$ or $x=q^2$ consists of the complement of all the lines through an affine point or all the lines in AG($3,q$) respectively.
\end{Opm}

\section{Number of possible parameters $x$ for CL line classes in AG($3,q$)}
Consider again Equation (\ref{CL1}) which is equivalent to
\begin{equation}\label{CL2}
x(x-1)\equiv 0 \mod 2(q+1).
\end{equation}
So we conclude, by Corollary \ref{GevKlausMain}, that if a CL line class \(\mathcal{L}\) with parameter \(x \in \{0,1,...,q^2\}\) exists in AG($3,q$), then Equation (\ref{CL2}) holds for \(x\).
Now if we consider the prime factorization \(2(q+1)=p_1^{h_1} \cdots  p_s^{h_s}\), then we can take a look at the system of equations
\begin{equation}\label{CL3}
x(x-1)\equiv 0 \mod p_i^{h_i}, \text{   for } i \in \{1,...,s\}.
\end{equation}
We know that if we find a value $x$ that satisfies this system of equations, we get, due to the fact that all the prime powers are coprime, a solution for (\ref{CL2}). In fact the Chinese Remainder Theorem states that this solution will be unique modulo $2(q+1)$.

\begin{Le}\label{Claim}
	Let $\mathcal{L}$ be a CL line class with parameter $x$ in AG($3,q$), where the prime factorization of $2(q+1)$ is equal to $p_1^{h_1} \cdots  p_s^{h_s}$. Then 
 \begin{equation}\label{CL4}
	x\equiv 0 \pmod{ p_i^{h_i}} \text{ or } x\equiv 1 \pmod{p_i^{h_i}}, \text{   for } i \in \{1,...,s\}.
	\end{equation}
\end{Le}
\begin{proof}
This follows from Equations (\ref{CL2}) and (\ref{CL3}).
\end{proof}

This lemma will enable us to find an upper bound for the number of parameters $x$ of possible CL line classes in AG($3,q$).
\begin{St}\label{theo1}
Let $x$ be the parameter of a CL line class $\mathcal{L}$ in AG($3,q$), so $x \leq q^2$. Consider now the prime factorization \(2(q+1)=p_1^{h_1}\cdots p_s^{h_s}\). Then there are at most
	
	\begin{equation}
	\left\{
	\begin{array}{ll}
	2^{s-1}q, & \text{if } q \text{ is even}\\
	2^{s-1}q-2^{s-1}+2, & \text{if } q \text{ is odd}
	\end{array} \right. ,
	\end{equation}
	possibilities for \(x\).
\end{St}
\begin{proof}
	If $\mathcal{L}$ is a CL line class with parameter $x$ in AG($3,q$), then, by Corollary \ref{GevKlausMain}, it follows that Equation (\ref{CL2}) holds. So to count the maximal number of possible parameters, we need to count the maximal number of solutions for (\ref{CL2}).
	
	Due to our previous observations about the Chinese Remainder Theorem and Lemma \ref{Claim}, we only need to count the number of possible solutions for the equations

	\[x\equiv 0,1 \text{        mod}\,\, p_i^{h_i},\]
	for every \(i \in \{1,...,s\}\). If we now pick in every equation a \(1\) or \(0\), then we know that there is a unique solution for
	\[x(x-1)\equiv 0 \mod 2(q+1).\]
	Note that there are \(2^s\) possibilities to pick such a solution. But remark that these  solutions are considered in the interval \(I=[0, 2(q+1)-1]\) and adding $2(q+1)$ to a solution gives a new solution. So one can ask how many times the interval $I$ fits inside the interval \([0,q^2]\). This  is equal to the number \(\bigg\lfloor\frac{q^2}{2(q+1)}\bigg\rfloor\).
	\begin{enumerate}
		\item For \(q\equiv 1 \pmod 2\):
		\[\frac{q^2}{2(q+1)}=\frac{q-1}{2}+\frac{1}{2(q+1)},\]
		where \(\frac{q-1}{2} \in \mathbb{N}\), since \(q\) is odd. This gives  that \(I\) fits $\frac{q-1}{2}$ times inside \([0,q^2]\). So in each of the following intervals, there is precisely one solution for every choice we made before
		\[[0,2(q+1)-1],[2(q+1),4(q+1)-1],...,\bigg[\left( \frac{q-1}{2}-1\right) 2(q+1),\left( \frac{q-1}{2}\right) 2(q+1)-1\bigg],\]
		where the last interval in this row can be simplified as follows
		
		\[\bigg[\left( \frac{q-1}{2}-1\right) 2(q+1),\left( \frac{q-1}{2}\right) 2(q+1)-1\bigg]=[q^2-2q-3,q^2-2].\]
		This all gives at most \(2^s\big(\frac{q-1}{2}\big)\) solutions in the first $\frac{q-1}{2}$ intervals.
		So now we only need to add the solutions for \(q^2-2< x\leq q^2\). For \(x=q^2-1\), there is, by Remark \ref{HogeParameters}, only one CL line class: the complement of all lines through an affine point. If $x=q^2$, we see  that the only possibility is to consider every line in AG(3,$q$). So we get in total
		\[2^s\bigg(\frac{q-1}{2}\bigg)+2=2^{s-1}q-2^{s-1}+2\]
		solutions for \(x \in [0,q^2]\).
		
		\item For \(q\equiv0 \pmod 2\), we get:
		\[\frac{q^2}{2(q+1)}=\frac{q-2}{2} +\frac{q+2}{2(q+1)}.\]
		Now  \(\frac{q-2}{2} \in \mathbb{N}\), since \(q\) is  even. So we get for all \(x\) in one of these intervals at most \(2^s\big(\frac{q-2}{2}\big)\) solutions. One can calculate that the last interval that fits inside $[0,q^2]$ is of the form \([q^2-3q-4,q^2-q-3]\). So, for \(q^2-q-3<x\leq q^2\), we can only estimate that there are at most \(2^s\) solutions. We get at most 
		\[2^s\bigg(\frac{q-2}{2}\bigg)+2^s=2^{s-1}q\]
		solutions for \(x \in [0,q^2]\).
		
	\end{enumerate}

\end{proof}
\section{Classification of CL line classes in AG($3,q$), $q\leq 4$}
In this section we want to classify all CL line classes in AG($3,q$), with $q\leq 4$. To do this, we will mainly use Corollary \ref{GevKlausMain}. But first we give some trivial examples of CL line classes in AG($3,q$).
\begin{Def}\label{TrivAGEX}
Consider AG($3,q$) and let us recall the CL line classes from Corollary \ref{SmallClas} together with their complements. As a consequence of this corollary and Remark \ref{HogeParameters}, these examples are in fact the only possible examples for their corresponding parameter $x\in \{0,1,q^2-1,q^2\}$. We call these examples of CL line classes \emph{trivial examples}.
\end{Def}

\begin{St}
The only CL line classes in AG($3,q$), with $q\in\{2,3\}$, are trivial.
\end{St}
\begin{proof}
As a consequence of Corollary \ref{GevKlausMain} and Lemma \ref{PropCLLC} (1), we obtain that the remaining possibilities for parameters of CL line classes in AG($3,q$) in these cases are $x\in \{0,1, q^2-1,q^2\}$. But for all of these parameters there only exist trivial examples, which proves the theorem.
\end{proof}
Secondly, we observe the case for AG($3,4$). Therefore, we will need the following theorem.
\begin{St}\label{ParamPG(3,4)}
Consider the space AG($3,4$).
\begin{enumerate}
\item Then there exists no CL line class of parameter $x=5$.
\item If there exists a CL line class $\mathcal{L}$ of parameter $x=6$, then $\mathcal{L}$ intersects every plane in $3 \pmod{5}$ lines.
\end{enumerate}
\end{St}
\begin{proof}
This theorem is proven for PG($3,4$) in \cite[Theorem 1.3]{GovPent}, so the case for AG($3,4$) follows from Theorem \ref{AffCLBasic}. Since every CL line class in AG($3,4$) is a CL line class in PG($3,4$), it follows that it satisfies the same conditions as a CL line class in PG($3,4$).
\end{proof}
This leads to the following result.
\begin{St}
The only CL line classes in AG($3,4$) are trivial.
\end{St}
\begin{proof}
Again we start with Corollary \ref{GevKlausMain} and Lemma \ref{PropCLLC} (1), which implies that for the parameter $x$ of a CL line class it holds that $x\in \{0,1,5,6,11,15,16\}$. But due to Theorem \ref{ParamPG(3,4)}, we can exclude $x=5$. Using this same theorem in combination with Lemma \ref{NoBetter}, we find that for a CL line class of parameter $x=6$ and an arbitrary plane $\pi$, it holds that
$$|\mathcal{L}\cap line(\pi)|\equiv 0\not\equiv 3 \mod 5.$$
Hence, the parameter $x=6$ is excluded. From Lemma \ref{PropCLLC} (2), it would then follow that the parameters $x\in\{q^2-6=10, q^2-5=11\}$ are excluded as well. The only remaining parameters are those of the trivial examples and we may conclude the assertion.
\end{proof}

\section{A non-trivial example and a consequence for AG($3,5$)}
Here we give a non-trivial example of a CL line class in AG($3,q$). We will use the example of De Beule, Demeyer, Metsch and Rodgers stated in \cite{DeBeule} and the example of Feng, Momihara and Xiang stated in \cite{Feng}. Both articles simultaneously found a CL line class with parameter $x=\frac{q^2-1}{2}$ in PG($3,q$) that is skew to the set of all lines in a plane $\pi$ (for $q \equiv 5,9 \pmod{12}$).  
\begin{St}\cite[Theorem 5.1 and Lemma 6.1]{DeBeule}
There exists a CL line class of parameter $x=\frac{q^2-1}{2}$ in PG($3,q$), for $q \equiv 5 \text{ or }9 \pmod{12}$, which is skew to the set of  lines in  a plane.
\end{St}
This example restricted to AG($3,q$), where we choose the plane at infinity as $\pi$, gives us by Theorem \ref{CLAf3} a CL line class in AG($3,q$) with the same parameter. So we can conclude the following corollary.
\begin{Gev}\label{DeBeuleVBGev}
	Consider the affine space AG($3,q$).  If \(q \equiv 5\) or \(9 \pmod {12}\), then there exists a CL line class with parameter \(x= \frac{q^2-1}{2}\).
\end{Gev}

\begin{Opm}
	This corollary, together with Lemma \ref{PropCLLC} (2), also gives that for \(q \equiv 5\) or \(9 \pmod {12}\), there exists a CL line class  with parameter \(x= \frac{q^2+1}{2}\) in AG($3,q$).
\end{Opm}

Now we want to give a complete characterization of the parameters for CL line classes in AG($3,5$). This result was based on a similar result for PG($3,5$) in \cite{MetschAndGavrilyuk}, where they used Theorem \ref{MetschRes}. In a similar way we will use Corollary \ref{GevKlausMain} to achieve this for AG($3,5$). Let us first state the result found in \cite{MetschAndGavrilyuk}.
\begin{St}\cite[Theorem 1.3]{MetschAndGavrilyuk}\label{q=5}
	A CL line class with parameter \(x\) exists in PG($3,5$) if and only if \(x \in \{0,1,2,10,12,13,14,16,24,25,26\}\).
\end{St}
\begin{Gev}
There exists a CL line class $\mathcal{L}$ of parameter $x$ in AG($3,5$) if and only if $x \in \{0,1,12,13, 24,25\}$.
\end{Gev}
\begin{proof}
Due to Theorem \ref{AffCLBasic}, every CL line class in AG($3,5$) defines a CL line class in PG($3,5$) with the same parameter. Hence, if there exists no CL line class of parameter $x$ in PG\((3,5)\), then there exists no CL line class of parameter $x$ in AG($3,5$). 

From Theorem \ref{q=5}, it follows that \(x\in \{0,1,2,10,12,13,14,16,24,25,26\}\) are the only possible parameters for CL line classes in AG($3,5$). But if we consider Corollary \ref{GevKlausMain}, with $2(q+1)=4\cdot 3$, this then reduces to \(x \in \{0,1,12,13,16,24,25\}\). 
	Notice that if there would exist a CL line class of  parameter \(x=16\), then it follows from Lemma \ref{PropCLLC} that the complement of this line class is a CL line class of parameter $x=q^2-16=9$. This parameter does not occur in the list of remaining possible parameters, so we find a contradiction.

We only need to show that all these cases occur. We list the following examples:
\begin{itemize}
\item $x=0$: Put $\mathcal{L}=\emptyset.$
\item $x=1$: Let $\mathcal{L}$ be the set of all the lines through a fixed affine point.
\item $x=12$:  See Corollary \ref{DeBeuleVBGev}.
\item $x \in \{13,24,25\}$: Use Lemma \ref{PropCLLC}, which states that the complement of a CL line class is also a CL line class.
\end{itemize}


\end{proof}
\section{Final remarks}

In this last section, we want to compare our results to the result obtained by Penttila in \cite{PenttilaThesis}. In his PhD thesis, Penttila not only considered CL line classes in PG($3,q$), but also observed symmetrical tactical decompositions in PG($3,q$). Such a symmetrical tactical decomposition is a pair $(P,L)$, with $P$ and $L$ a specific partition of the points and lines of PG($3,q$) respectively. For more information, we refer to \cite{Cameron-Liebler}. Here it was proven that every line class of the partition $L$ defines a CL line class in PG($3,q$). A similar case follows for AG($3,q$). Besides this fact, Penttila also observed that if a symmetrical tactical decomposition of PG($3,q$) contains the hyperplane at infinity as a line and point class, then this symmetrical decomposition is also a symmetrical tactical decomposition in AG($3,q$). This result is comparable with Theorem \ref{CLAf3}. Conversely, every symmetrical tactical decomposition in AG($3,q$) can be extended to a symmetrical tactical decomposition in PG($3,q$). This fact is comparable with Theorem \ref{AffCLBasic}. Hence, these results are not exactly the same as Theorems \ref{CLAf3} and \ref{AffCLBasic}, since not every CL line class in PG($3,q$) (and possibly in AG($3,q$)) is a line class in a symmetrical tactical decomposition. Hence, it remains interesting to find such analogous results.
\paragraph{Acknowledgement}
The research of Jozefien D'haeseleer is supported by the FWO (Research Foundation Flanders).

The research of Andrea \v{S}vob is supported by the Croatian Science Foundation under the project 6732.

The authors thank the referees for their suggestions to improve this article.






\end{document}